\documentclass[10pt,reqno]{amsart}

\usepackage{amssymb}
\usepackage{amsmath}
\usepackage{amsfonts}
\usepackage{amscd}
\usepackage{mathrsfs}
\usepackage{bigints}
\usepackage{url}

\newcommand{\C}{\mathbf{C}}
\newcommand{\F}{\mathfrak{F}}
\newcommand{\Fa}{\mathfrak{F}_a}

\newcommand{\Z}{\mathscr{Z}}

\newtheorem{theorem}{Theorem}[section]
\newtheorem{corollary}[theorem]{Corollary}

\newtheorem{proposition}[theorem]{Proposition}
\newtheorem*{example}{\sl Example}

\newtheorem*{remark}{\sl Remark}
\numberwithin{equation}{section}

\begin{document}
\title[]{On generalized Fermat Diophantine functional and partial differential equations in $\C^2$}

\author[]{Wei Chen$^\ddag$, Qi Han$^{\star,\dag}$, Qiong Wang$^\ddag$}

\address{\footnotesize{$^{\ddag}$School of Sciences, Chongqing University of Posts \& Telecommunications, Chongqing 400065, China
\vskip 2pt $^\dag$Department of Mathematics, Texas A\&M University-San Antonio, San Antonio, Texas 78224, USA
\vskip 2pt Email: {\sf weichensdu@126.com} (W. Chen)} {\sf qhan@tamusa.edu} (Q. Han) {\sf qiongwangsdu@126.com} (Q. Wang)}

\thanks{$^\star$Qi Han is the corresponding author of this research work.}




\begin{abstract}
In this paper, we characterize meromorphic solutions $f(z_1,z_2),g(z_1,z_2)$ to the generalized Fermat Diophantine functional equations $h(z_1,z_2)f^m+k(z_1,z_2)g^n=1$ in $\C^2$ for integers $m,n\geq2$ and nonzero meromorphic functions $h(z_1,z_2),k(z_1,z_2)$ in $\C^2$.
Meromorphic solutions to associated partial differential equations are also studied.
\end{abstract}

\maketitle

\section{Introduction}\label{Int}
In 1995, Andrew Wiles \cite{Wi, WT} finally proved the profound Fermat's Last Theorem, remained open for over 350 years, which states, in a simple-looking form, that the equations $a^n+b^n=c^n$ have no positive integral solution if $n\geq3$.
An equivalent form of this prestigious result is that the equations $x^n+y^n=1$ have no positive rational solution if $n>2$.

It seems that Montel \cite{Mo} in 1927 first studied the functional equations $f^n+g^n=1$ analogous to Fermat equations $x^n+y^n=1$, and he observed that all entire solutions $f,g$ must be constant if $n>2$; see also Jategaonkar \cite{Ja}.
Later work were proved by Baker \cite{Ba} and Gross \cite{Gr1,Gr2}, where they showed that all meromorphic solutions to $f^n+g^n=1$ must be constant if $n>3$ and also characterized associated nonconstant meromorphic solutions for $n=2,3$ in $\C$.

Cartan \cite{Ca} in 1933 first considered the Fermat Diophantine functional equations $f^m+g^n=1$, and he observed that all entire solutions $f,g$ must be constant if $mn>m+n$; see also Section 4 of Gundersen and Hayman \cite{GH}.
We are not aware of any earlier work after that, and it seems to us Li \cite[Section 4]{Li12} in 2012 first investigated meromorphic solutions to $f^m+g^n=1$, where he derived that all meromorphic solutions must be constant if $mn>\max\{2m+n,m+2n\}$ and further analyzed associated nonconstant meromorphic solutions for the remaining cases, which was recently complemented by Chen, Han and Liu \cite[Proposition 1]{CHL}.

In a different context, Khavinson \cite{Kh1} in 1995 derived the first result on entire solutions $u$ to the classical 2d eikonal equation $u^2_{z_1}+u^2_{z_2}=1$ in $\C^2$ and proved that $u$ must be affine; see also Saleeby \cite{Sa1}.
Khavinson's work inspired several results such as Li \cite{Li05TAMS}, Hemmati \cite{He}, Saleeby \cite{Sa2,Sa3} and Chen and Han \cite{CH,Ha} on complex analytic solutions to variants of $u^2_{z_1}+u^2_{z_2}=1$ in $\C^2$, and notably provided some insight to the work of Caffarelli and Crandall \cite[Remark 2.3]{CC}.
One may also consult Khavinson and Lundberg \cite{KL}, and Khavinson \cite{Kh2}.

The first order nonlinear partial differential equation $u^2_{z_1}+u^2_{z_2}=1$ is a special form of the functional equation $f^2+g^2=1$ in $\C^2$ that admits of nonconstant complex analytic solutions such as
$\biggl\{\begin{array}{ll}
f=\sin(\alpha(z_1,z_2))\medskip\\
g=\cos(\alpha(z_1,z_2))\end{array}$
(entire) and
$\biggl\{\begin{array}{ll}
f=\sec(\alpha(z_1,z_2))\medskip\\
g=\texttt{i}\tan(\alpha(z_1,z_2))\end{array}$
(meromorphic) for an entire function $\alpha(z_1,z_2)$ in $\C^2$.
Noting that $(u_{z_1})_{z_2}=(u_{z_2})_{z_1}$ for complex analytic functions, Li (with Ye) in a series of papers \cite{LY,Li08,Li14} studied the conditions under which there is no nonconstant meromorphic solution to the Fermat Diophantine functional equations $f^m+g^n=1$ for $m,n\geq2$ in $\C^2$, and the condition that $\Z(f_{z_2})=\Z(g_{z_1})$ counting multiplicities was identified with $\Z(\beta)$ denoting the zero set of a complex analytic function $\beta$.

When a similar inquiry on the nonexistence of nonconstant meromorphic solutions to $f^m+g^n=1$ for $m,n\geq2$ in $\C$ arises, Chen, Han and Liu \cite[Theorem 4]{CHL} recently provided (weaker) conditions such as $\Z(f')=\Z(g')$ ignoring multiplicities if $m=n=2$, but $\Z(f')\subseteq\Z(g')$ or $\Z(g')\subseteq\Z(f')$ ignoring multiplicities if $m,n\geq2$ and $(m,n)\neq(2,2)$.

Li \cite{Li05FM,Li18} studied the generalized Fermat Diophantine functional equations
\begin{equation}\label{Eq1.1}
h(z_1)f^m+k(z_2)g^n=1
\end{equation}
and proved the result below on associated complex analytic solutions $f(z_1,z_2),g(z_1,z_2)$ in $\C^2$.

\begin{theorem}\label{Thm1.1}
Assume that $h(z_1),k(z_2)$ are two nonzero meromorphic functions in $\C$, and $m,n\geq2$ are two integers.
Then, one has the results as follows.\\
{\bf(A)} All meromorphic solutions $f(z_1,z_2),g(z_1,z_2)$ to \eqref{Eq1.1} in $\C^2$, with $(m,n)\neq(2,2)$, satisfy
\begin{equation}\label{Eq1.2}
T(r,f)+T(r,g)=O(T(r,h)+T(r,k))
\end{equation}
outside a set of $r$ of finite Lebesgue measure, provided $\Z(f_{z_2})=\Z(g_{z_1})$ ignoring multiplicities.\\
{\bf(B)} All entire solutions $f(z_1,z_2),g(z_1,z_2)$ to \eqref{Eq1.1} in $\C^2$, with $m=n=2$, satisfy \eqref{Eq1.2} again outside a set of $r$ of finite Lebesgue measure, provided $\Z(f_{z_2})=\Z(g_{z_1})$ counting multiplicities.
\end{theorem}

If $h(z_1)\equiv k(z_2)\equiv1$, then all complex analytic solutions are reduced to constant, consistent with the proceeding work \cite{LY,Li08,Li14}.
Nevertheless, the functions $h(z_1),k(z_2)$ cannot be readily replaced by functions $h(z_1,z_2),k(z_1,z_2)$ in $\C^2$ without modifying conditions.
An example was given in \cite[Remark 2.2]{Li05FM}, and below is another one very much akin to that.

\begin{example}
One has $h(z_1,z_2)=\frac{z_1^2}{z_2^2}$, $k(z_1,z_2)=z_2^2$, $f(z_1,z_2)=\frac{z_2}{z_1}\sin z_1$ and $g(z_1,z_2)=\frac{1}{z_2}\cos z_1$ satisfying
$h(z_1,z_2)f^2+k(z_1,z_2)g^2=1$ with $\Z(f_{z_2})=\Z(g_{z_1})$ counting multiplicities, but the estimate \eqref{Eq1.2} apparently fails.
\end{example}

It is natural to seek conditions in which one can use $h(z_1,z_2),k(z_1,z_2)$ in $\C^2$ as coefficients in \eqref{Eq1.1} in place of $h(z_1),k(z_2)$ in $\C$.
Li \cite[Section 4]{Li12} has already observed the estimate \eqref{Eq1.2} for meromorphic solutions $f(z_1,z_2),g(z_1,z_2)$ in $\C^2$ to
\begin{equation}\label{Eq1.3}
h(z_1,z_2)f^m+k(z_1,z_2)g^n=1
\end{equation}
when $mn>\max\{2m+n,m+2n\}$, and for entire solutions $f(z_1,z_2),g(z_1,z_2)$ in $\C^2$ to \eqref{Eq1.3} when $mn>m+n$, with no constraint on the zero sets $\Z(f_{z_2}),\Z(g_{z_1})$.

This paper, inspired by Li (with Ye, and L\"{u}) \cite{Li05FM,LY,Li08,Li12,Li14,Li18,LL}, aims at recognizing conditions in which meromorphic solutions to \eqref{Eq1.3} in $\C^2$ satisfy \eqref{Eq1.2} for the remaining cases of $mn>\max\{2m+n,m+2n\}$.
We mainly focus on $(m,n),(n,m)=(2,2),(2,3),(2,4),(3,3)$, and also provide a general counterexample for the cases $m=1,n\geq2$.

Before to proceed, let $\F$ be the set of meromorphic functions $f$ in $\C^2$ such that $\Theta(\infty,f)>0$, where $\Theta(\infty,f)$ denotes a deficient value of $f$ at $\infty$; see Hayman \cite[Section 2.4]{Hay1}.
Notice if $f$ is entire, then $\Theta(\infty,f)=1$.
So, all entire functions are members of $\F$.

Our main results of this paper can be formulated as follows.

\begin{theorem}\label{Thm1.2}
Assume that $h(z_1,z_2),k(z_1,z_2)$ are two nonzero meromorphic functions in $\C^2$, and $m,n\geq2$ are two integers.
Then, one has the results as follows.\\
{\bf(A)} All meromorphic solutions $f(z_1,z_2),g(z_1,z_2)$ to \eqref{Eq1.3} in $\C^2$, with $(m,n)\neq(2,2)$, satisfy the estimate \eqref{Eq1.2} outside a set of $r$ of finite Lebesgue measure, provided
\begin{equation}\label{Eq1.4}
\Z\big((hf^m)_{z_2}/(mhf^{m-1})\big)=\Z\big((kg^n)_{z_1}/(nkg^{n-1})\big)
\end{equation}
ignoring multiplicities.\\
{\bf(B)} All meromorphic solutions $f(z_1,z_2),g(z_1,z_2)$ to \eqref{Eq1.3} in $\C^2$, with $m=n=2$, again satisfy the estimate \eqref{Eq1.2} outside a set of $r$ of finite Lebesgue measure, provided
\begin{equation}\label{Eq1.5}
\Z\big((hf^2)_{z_2}/(2hf)\big)=\Z\big((kg^2)_{z_1}/(2kg)\big)
\end{equation}
ignoring multiplicities and either $f\in\F$ or $g\in\F$.
\end{theorem}

For $m=1,n\geq2$ (or similarly for $m\geq2,n=1$), one easily sees a family of counterexamples as follows.

\begin{example}
Assume $h(z_1,z_2),k(z_1,z_2)$ are meromorphic functions while $p(z_1,z_2)$ is a polynomial in $\C^2$.
Then, the functions
\begin{equation*}
f(z_1,z_2)=\frac{1}{h(z_1,z_2)}-\frac{k(z_1,z_2)}{h(z_1,z_2)}e^{np(z_1,z_2)}\,\,\text{and}\,\,g(z_1,z_2)=e^{p(z_1,z_2)}
\end{equation*}
satisfy $h(z_1,z_2)f+k(z_1,z_2)g^n=1$.
Let, say, $h(z_1,z_2)=k(z_1,z_2)=p(z_1,z_2)=z^2_1+2z_1z_2+z^2_2$ and observe that
\begin{equation*}
\Z\big((hf)_{z_2}/h\big)=\Z\big((kg^n)_{z_1}/(nkg^{n-1})\big)
\end{equation*}
counting multiplicities, but apparently \eqref{Eq1.2} fails.
\end{example}

\begin{corollary}\label{Cor1.3}
All meromorphic solutions $u(z_1,z_2)$ to the general 2d eikonal equations
\begin{equation}\label{Eq1.6}
u^m_{z_1}+u^n_{z_2}=1,
\end{equation}
with $m,n\geq2$, must be affine.
\end{corollary}

\begin{remark}
Corollary \ref{Cor1.3} (also) follows from Li \cite{Li08,LY,Li14,Li18}.
The condition \eqref{Eq1.4} is not entirely new, which was first mentioned in \cite{Li18} and used in \cite{LL} but only for entire solutions to \eqref{Eq1.3} for $m=n=2$; note that the inclusion condition in \cite[Theorem 2.1]{LL} will not lead to its associated conclusions, and the issue comes from the analysis at lines 4-5 of its page 7.
\end{remark}

In Section 2, we prove Theorem \ref{Thm1.2}.
In Section 3, we further discuss Theorem \ref{Thm1.2} {\bf(B)} about the family $\F$, study an extension of Corollary \ref{Cor1.3} on the first order nonlinear partial differential equations $h(z_1)u^m_{z_1}+k(z_2)u^n_{z_2}=1$, and characterize entire solutions to the first order nonlinear partial differential equations $h(z_1)u^m_{z_1}+k(z_2)u_{z_2}=1$ and $h(z_1)u_{z_1}+k(z_2)u^n_{z_2}=1$ for $m,n\geq2$.

To prove our results, we will need Nevanlinna theory; see the classical monograph Stoll \cite{St}.
So, we assume the familiarity with the basics of this theory such as the characteristic function $T(r,f)$, the proximity function $m(r,f)$, the counting functions $N(r,f)$ (counting multiplicities) and $\overline{N}(r,f)$ (ignoring  multiplicities), and the first and second main theorems.
$S(r,f)$ denotes a quantity with $S(r,f)=o(T(r,f))$ as $r\to+\infty$ outside a set of $r$ of finite Lebesgue measure.
Define $\Theta(\infty,f):=1-\limsup\limits_{r\to\infty}\frac{\overline{N}(r,f)}{T(r,f)}$.
Set $\operatorname{div}^0_f(\xi_0)$ to be the multiplicity of $f$ at a zero $\xi_0$, being the smallest degree of the homogeneous polynomials appearing in the Taylor series expansion of $f$ around $\xi_0$; as a common practice, write $\operatorname{div}^0_f(\xi_0):=+\infty$ when $f\equiv0$.

The problems studied in this paper are independent of Navanlinna theory.
So, it is of great interest to reprove the results here without involving this heavy machinery.

\section{Proof of Theorem \ref{Thm1.2}}\label{PMT}
\begin{proof}[Proof of Theorem \ref{Thm1.2} (A)]
Set $U:=hf^m$ and $V:=kg^n$.
Then, it follows from \eqref{Eq1.3} that
\begin{equation}
U+V=1,\nonumber
\end{equation}
and hence one has
\begin{equation}\label{Eq2.1}
U_{z_j}+V_{z_j}=0
\end{equation}
for $j=1,2$.
Noticing
\begin{equation}\label{Eq2.2}
\begin{split}
&U_{z_j}=h_{z_j}f^m+mhf^{m-1}f_{z_j}=f^{m-1}(h_{z_j}f+mhf_{z_j})\,\,\text{and}\\
&V_{z_j}=k_{z_j}g^n+nkg^{n-1}g_{z_j}=g^{n-1}(k_{z_j}g+nkg_{z_j}),
\end{split}
\end{equation}
we define
\begin{equation}\label{Eq2.3}
P:=\frac{F^m}{f^{2m}}\frac{G^n}{g^{2n}}
\end{equation}
for
\begin{equation}
F:=h_{z_2}f+mhf_{z_2}=\frac{U_{z_2}}{f^{m-1}}\,\,\text{and}\,\,G:=k_{z_1}g+nkg_{z_1}=\frac{V_{z_1}}{g^{n-1}}.\nonumber
\end{equation}

Using \eqref{Eq1.3}, we have
\begin{equation}\label{Eq2.4}
mT(r,f)=nT(r,g)+O(T(r,h)+T(r,k)).
\end{equation}
Applying the lemma on logarithmic derivative (see Vitter \cite{Vi}, and Biancofiore and Stoll \cite{BS}), we deduce from \eqref{Eq1.3} and \eqref{Eq2.2}-\eqref{Eq2.4} that
\begin{equation}\label{Eq2.5}
\begin{split}
&\hspace{4.66mm}m(r,P)=m\bigg(r,\frac{F^m}{f^{2m}}\frac{G^n}{g^{2n}}\bigg)=m\bigg(r,\frac{kF^m}{f^m(1-hf^m)}\frac{hG^n}{g^n(1-kg^n)}\bigg)\\
&\leq m\bigg(r,\frac{F^m}{f^m(1-hf^m)}\bigg)+m\bigg(r,\frac{G^n}{g^n(1-kg^n)}\bigg)+T(r,h)+T(r,k)+O(1)\\
&=m\bigg(r,\frac{F^m}{f^m}+\frac{hF^m}{1-hf^m}\bigg)+m\bigg(r,\frac{G^n}{g^n}+\frac{kG^n}{1-kg^n}\bigg)+T(r,h)+T(r,k)+O(1)\\
&=m\bigg(r,\frac{F^m}{f^m}\bigg)+m\bigg(r,\frac{hF^m}{1-hf^m}\bigg)+m\bigg(r,\frac{G^n}{g^n}\bigg)+m\bigg(r,\frac{kG^n}{1-kg^n}\bigg)\\
&\quad+T(r,h)+T(r,k)+O(1)\\
&\leq m\bigg\{m\bigg(r,\frac{f_{z_2}}{f}\bigg)+T(r,h)+T(r,h_{z_2})\bigg\}+m\bigg(r,\frac{F^m}{(1-hf^m)_{z_2}}\frac{(1-hf^m)_{z_2}}{1-hf^m}\bigg)\\
&\quad+n\bigg\{m\bigg(r,\frac{g_{z_1}}{g}\bigg)+T(r,k)+T(r,k_{z_1})\bigg\}+m\bigg(r,\frac{G^n}{(1-kg^n)_{z_1}}\frac{(1-kg^n)_{z_1}}{1-kg^n}\bigg)\\
&\quad+O(T(r,h)+T(r,k))\\
&\leq m\bigg(r,\frac{F^m}{U_{z_2}}\frac{U_{z_2}}{1-U}\bigg)+m\bigg(r,\frac{G^n}{V_{z_1}}\frac{V_{z_1}}{1-V}\bigg)\\
&\quad+S(r,f)+S(r,g)+O(T(r,h)+T(r,k))\\
&=m\bigg(r,\frac{F^{m-1}}{f^{m-1}}\bigg)+m\bigg(r,\frac{U_{z_2}}{1-U}\bigg)+m\bigg(r,\frac{G^{n-1}}{g^{n-1}}\bigg)+m\bigg(r,\frac{V_{z_1}}{1-V}\bigg)\\
&\quad+S(r,f)+S(r,g)+O(T(r,h)+T(r,k))\\
&\leq(m-1)\bigg\{m\bigg(r,\frac{f_{z_2}}{f}\bigg)+T(r,h)+T(r,h_{z_2})\bigg\}+S(r,U)\\
&\quad+(n-1)\bigg\{m\bigg(r,\frac{g_{z_1}}{g}\bigg)+T(r,k)+T(r,k_{z_1})\bigg\}+S(r,V)\\
&\quad+S(r,f)+S(r,g)+O(T(r,h)+T(r,k))\\
&=S(r,f)+S(r,g)+O(T(r,h)+T(r,k)),
\end{split}
\end{equation}
as clearly $T(r,U)=O(T(r,f)+T(r,h))$ and $T(r,V)=O(T(r,g)+T(r,k))$.

{\bf Claim 1.} $P$ is holomorphic at each pole of $f$ or $g$, which is neither a zero nor a pole of $h,k$ and their partial derivatives.
To prove this claim, we rewrite $P$ as
\begin{equation}
\begin{split}
P&=\frac{F^m}{f^{2m}}\frac{G^n}{g^{2n}}=\bigg(\frac{h_{z_2}}{f}+mh\frac{f_{z_2}}{f^2}\bigg)^m\bigg(\frac{k_{z_1}}{g}+nk\frac{g_{z_1}}{g^2}\bigg)^n\\
&=(h_{z_2}\tilde{f}-mh\tilde{f}_{z_2})^m(k_{z_1}\tilde{g}-nk\tilde{g}_{z_1})^n\nonumber
\end{split}
\end{equation}
for $\tilde{f}:=\frac{1}{f}$ and $\tilde{g}:=\frac{1}{g}$.
We observe from \eqref{Eq1.3} that a pole $\xi_\infty$ of $f$ with multiplicity $l$, which is neither a zero nor a pole of $h,k$ and their partial derivatives, is also a pole of $g$ with multiplicity $ml/n\geq1$, an integer.
Thus, we have
\begin{equation}\label{Eq2.6}
\operatorname{div}^0_P(\xi_\infty)\geq m(l-1)+n\big(ml/n-1\big)=2ml-m-n\geq0,
\end{equation}
so that Claim 1 follows from \eqref{Eq2.6} and a symmetric analysis on poles of $g$.

{\bf Claim 2.} At each zero of $f$ or $g$, which is neither a zero nor a pole of $h,k$ and their partial derivatives, we have
\begin{equation}
\operatorname{div}^0_P(\xi_0)\geq1.\nonumber
\end{equation}
We shall only discuss on a zero $\xi_0$ of $f$ and consider two cases separately.

{\bf(1)} When $\operatorname{div}^0_f(\xi_0)=1$, then we know from the first equation in \eqref{Eq2.2} that $U_{z_j}(\xi_0)=0$ with multiplicity at least $m-1$, which immediately leads to $V_{z_j}(\xi_0)=0$ with multiplicity at least $m-1$ by \eqref{Eq2.1}.
Noticing $g(\xi_0)\neq0$ by \eqref{Eq1.3} and recalling the hypothesis $\Z\big(U_{z_2}/(mhf^{m-1})\big)=\Z\big(V_{z_1}/(nkg^{n-1})\big)$ ignoring multiplicities, one derives that $\xi_0$ must also be a zero of $U_{z_2}/f^{m-1}$; that is, $\xi_0$ is a zero of $F$ by definition.
In view of \eqref{Eq2.3}, we have
\begin{equation}\label{Eq2.7}
\begin{split}
\operatorname{div}^0_P(\xi_0)&=\operatorname{div}^0_{F^m}(\xi_0)+\operatorname{div}^0_{G^n}(\xi_0)-2\operatorname{div}^0_{f^m}(\xi_0)\\
&=m\operatorname{div}^0_{\frac{U_{z_2}}{f^{m-1}}}(\xi_0)+n\operatorname{div}^0_{V_{z_1}}(\xi_0)-2m\operatorname{div}^0_f(\xi_0)\\
&\geq m+n(m-1)-2m.
\end{split}
\end{equation}
When $n=2$, then our assumptions $m\geq2$ and $(m,n)\neq(2,2)$ yield $m\geq3$; so, \eqref{Eq2.7} implies
\begin{equation}
\operatorname{div}^0_P(\xi_0)\geq2(m-1)-m=m-2\geq1.\nonumber
\end{equation}
When $n\geq3$, then $m\geq2$ and it follows from \eqref{Eq2.7} that
\begin{equation}
\operatorname{div}^0_P(\xi_0)\geq3(m-1)-m=2m-3\geq1.\nonumber
\end{equation}

{\bf (2)} When $\operatorname{div}^0_f(\xi_0)=l\geq2$, then we derive from the first equation in \eqref{Eq2.2} that $U_{z_j}(\xi_0)=0$ with multiplicity at least $(m-1)l+l-1=ml-1$, which immediately yields $V_{z_j}(\xi_0)=0$ with multiplicity at least $ml-1$ by \eqref{Eq2.1}.
As $g(\xi_0)\neq0$ by \eqref{Eq1.3}, we see from \eqref{Eq2.3} that
\begin{equation}\label{Eq2.8}
\begin{split}
\operatorname{div}^0_P(\xi_0)&=\operatorname{div}^0_{F^m}(\xi_0)+\operatorname{div}^0_{G^n}(\xi_0)-2\operatorname{div}^0_{f^m}(\xi_0)\\
&=m\operatorname{div}^0_{\frac{U_{z_2}}{f^{m-1}}}(\xi_0)+n\operatorname{div}^0_{V_{z_1}}(\xi_0)-2m\operatorname{div}^0_f(\xi_0)\\
&\geq m(l-1)+n(ml-1)-2ml.
\end{split}
\end{equation}
When $n=2$, then our assumptions $m\geq2$ and $(m,n)\neq(2,2)$ yield $m\geq3$; so, \eqref{Eq2.8} implies
\begin{equation}
\operatorname{div}^0_P(\xi_0)\geq2(ml-1)-ml-m=ml-m-2\geq m(l-1)-2\geq1.\nonumber
\end{equation}
When $n\geq3$, then $m\geq2$ and it follows from \eqref{Eq2.8} that
\begin{equation}
\operatorname{div}^0_P(\xi_0)\geq3(ml-1)-ml-m=2ml-m-3\geq m(2l-1)-3\geq3.\nonumber
\end{equation}

So, Claim 2 follows from the preceding estimates and a symmetric analysis on zeros of $g$.

By virtue of \eqref{Eq2.3} and Claims 1-2, it is easy to deduce that
\begin{equation}
N(r,P)=O\bigg(N(r,h)+N\Big(r,\frac{1}{h}\Big)+N(r,k)+N\Big(r,\frac{1}{k}\Big)\bigg)\leq O(T(r,h)+T(r,k)),\nonumber
\end{equation}
which combined with \eqref{Eq2.5} leads to
\begin{equation}\label{Eq2.9}
T(r,P)\leq S(r,f)+S(r,g)+O(T(r,h)+T(r,k)).
\end{equation}

When $P\equiv0$, then we know from \eqref{Eq2.3} that $FG\equiv0$; therefore, either $F\equiv0$ or $G\equiv0$.
Our hypothesis $\Z\big(U_{z_2}/(mhf^{m-1})\big)=\Z\big(V_{z_1}/(nkg^{n-1})\big)$ yields $F\equiv G\equiv0$.
Thus, $U_{z_1},U_{z_2},V_{z_1},V_{z_2}$ must all be identically equal to zero by \eqref{Eq2.1}, so that $U=hf^m$ and $V=kg^n$ both are constant, which immediately leads to the conclusion of Theorem \ref{Thm1.2}.

When $P\not\equiv0$, then \eqref{Eq1.3} and Nevanlinna's first and second main theorems yield
\begin{equation}\label{Eq2.10}
\begin{split}
&\hspace{4.6mm}mT(r,f)=T(r,f^m)+O(1)\leq T(r,hf^m)+T(r,h)\\
&\leq\overline{N}(r,hf^m)+\overline{N}\Big(r,\frac{1}{hf^m}\Big)+\overline{N}\Big(r,\frac{1}{hf^m-1}\Big)+S(r,f)+T(r,h)\\
&\leq\overline{N}(r,f)+\overline{N}\Big(r,\frac{1}{f}\Big)+\overline{N}\Big(r,\frac{1}{kg^n}\Big)+S(r,f)+O(T(r,h))\\
&\leq\overline{N}(r,f)+\overline{N}\Big(r,\frac{1}{f}\Big)+\overline{N}\Big(r,\frac{1}{g}\Big)+S(r,f)+O(T(r,h)+T(r,k)),
\end{split}
\end{equation}
and similarly
\begin{equation}\label{Eq2.11}
nT(r,g)\leq\overline{N}(r,g)+\overline{N}\Big(r,\frac{1}{g}\Big)+\overline{N}\Big(r,\frac{1}{f}\Big)+S(r,g)+O(T(r,h)+T(r,k)).
\end{equation}

From Claim 2, we know from \eqref{Eq2.4} and \eqref{Eq2.9}-\eqref{Eq2.11} that
\begin{equation}
\begin{split}
&\hspace{4.6mm}2mT(r,f)=mT(r,f)+nT(r,g)+O(T(r,h)+T(r,k))\\
&\leq\overline{N}(r,f)+\overline{N}(r,g)+2\overline{N}\Big(r,\frac{1}{g}\Big)+2\overline{N}\Big(r,\frac{1}{f}\Big)+S(r,f)+S(r,g)+O(T(r,h)+T(r,k))\\
&\leq2\overline{N}(r,f)+4\overline{N}\Big(r,\frac{1}{P}\Big)+S(r,f)+S(r,g)+O(T(r,h)+T(r,k))\\
&\leq2\overline{N}(r,f)+4T(r,P)+S(r,f)+S(r,g)+O(T(r,h)+T(r,k))\\
&\leq2T(r,f)+S(r,f)+S(r,g)+O(T(r,h)+T(r,k)),\nonumber
\end{split}
\end{equation}
which together with the assumption $m\geq2$ yields the conclusion of Theorem \ref{Thm1.2}.
\end{proof}

\begin{proof}[Proof of Theorem \ref{Thm1.2} (B)]
Set $U=hf^2$ and $V=kg^2$ to have \eqref{Eq2.1}-\eqref{Eq2.2} as before.
Write
\begin{equation}\label{Eq2.12}
P_1:=\frac{F^2}{f^2}\frac{G^2}{g^4}=k\bigg(\frac{F^2}{f^2}-\frac{hF^2}{1-hf^2}\bigg)\frac{G^2}{g^2}
\end{equation}
for $F=h_{z_2}f+2hf_{z_2}=U_{z_2}/f$ and $G=k_{z_1}g+2kg_{z_1}=V_{z_1}/g$ as before.

Now, utilize \eqref{Eq1.3} to see
\begin{equation}\label{Eq2.13}
T(r,f)=T(r,g)+O(T(r,h)+T(r,k)),
\end{equation}
and apply the lemma on logarithmic derivative again, like in \eqref{Eq2.5}, to deduce that
\begin{equation}\label{Eq2.14}
m(r,P_1)=S(r,f)+S(r,g)+O(T(r,h)+T(r,k)).
\end{equation}

{\bf Claim 1.} $P_1$ has a pole at each pole of $f$ that is neither a zero nor a pole of $h,k$ and their partial derivatives, with multiplicity at most $2$.
To prove this claim, rewrite $P_1$ as
\begin{equation}
\begin{split}
P_1&=\frac{F^2}{f^2}\frac{G^2}{g^4}=\bigg(h_{z_2}+2h\frac{f_{z_2}}{f}\bigg)^2\bigg(\frac{k_{z_1}}{g}+2k\frac{g_{z_1}}{g^2}\bigg)^2\\
&=\frac{1}{\tilde{f}^2}(h_{z_2}\tilde{f}-2h\tilde{f}_{z_2})^2(k_{z_1}\tilde{g}-2k\tilde{g}_{z_1})^2\nonumber
\end{split}
\end{equation}
for $\tilde{f}=\frac{1}{f}$ and $\tilde{g}=\frac{1}{g}$.
We observe from \eqref{Eq1.3} that a pole $\xi_\infty$ of $f$ with multiplicity $l$, which is neither a zero nor a pole of $h,k$ and their partial derivatives, is also a pole of $g$ with the same multiplicity $l$.
Thus, we have $\operatorname{div}^0_{\tilde{f}^2P_1}(\xi_\infty)\geq4(l-1)\geq0$, so that $\operatorname{div}^0_{P_1}(\xi_\infty)\geq4(l-1)-2l=2(l-2)\geq0$ if $l\geq2$; if $l=1$, then one has $\operatorname{div}^0_{\frac{1}{P_1}}(\xi_\infty)\leq\operatorname{div}^0_{\tilde{f}^2}(\xi_\infty)=2$.

Because poles of $f$ and $g$ appear simultaneously, Claim 1 is thus proved.

{\bf Claim 2.} We have
\begin{equation}\label{Eq2.15}
\operatorname{div}^0_{P_1}(\xi_0)\geq2l
\end{equation}
at each zero of $f$ with multiplicity $l$, which is neither a zero nor a pole of $h,k$ and their partial derivatives.
We consider two cases separately.

{\bf(1)} When $\operatorname{div}^0_f(\xi_0)=1$, then we know from the first equation in \eqref{Eq2.2} that $U_{z_j}(\xi_0)=0$ with multiplicity at least $1$, which then yields $V_{z_j}(\xi_0)=0$ with multiplicity at least $1$ by \eqref{Eq2.1}.
Since $g(\xi_0)\neq0$ via \eqref{Eq1.3} and $\Z\big(U_{z_2}/f\big)=\Z\big(V_{z_1}/g\big)$ ignoring multiplicities via \eqref{Eq1.5}, $\xi_0$ is a zero of $F$ by definition.
In view of \eqref{Eq2.12}, we have
\begin{equation}
\begin{split}
\operatorname{div}^0_{P_1}(\xi_0)&=\operatorname{div}^0_{F^2}(\xi_0)+\operatorname{div}^0_{G^2}(\xi_0)-\operatorname{div}^0_{f^2}(\xi_0)\\
&=2\operatorname{div}^0_{\frac{U_{z_2}}{f}}(\xi_0)+2\operatorname{div}^0_{V_{z_1}}(\xi_0)-2\operatorname{div}^0_f(\xi_0)\geq2.\nonumber
\end{split}
\end{equation}

{\bf (2)} When $\operatorname{div}^0_f(\xi_0)=l\geq2$, then we derive from the first equation in \eqref{Eq2.2} that $U_{z_j}(\xi_0)=0$ with multiplicity at least $2l-1$, which immediately yields $V_{z_j}(\xi_0)=0$ with multiplicity at least $2l-1$ by \eqref{Eq2.1}.
As $g(\xi_0)\neq0$ by \eqref{Eq1.3}, we see from \eqref{Eq2.12} that
\begin{equation}
\begin{split}
\operatorname{div}^0_{P_1}(\xi_0)&=\operatorname{div}^0_{F^2}(\xi_0)+\operatorname{div}^0_{G^2}(\xi_0)-\operatorname{div}^0_{f^2}(\xi_0)\\
&=2\operatorname{div}^0_{\frac{U_{z_2}}{f}}(\xi_0)+2\operatorname{div}^0_{V_{z_1}}(\xi_0)-2\operatorname{div}^0_f(\xi_0)\\
&\geq2(l-1)+2(2l-1)-2l=2l+2(l-2)\geq2l.\nonumber
\end{split}
\end{equation}

{\bf Claim 3.} We have $\operatorname{div}^0_{P_1}(\xi_0)\geq0$ at each zero of $g$ that is neither a zero nor a pole of $h,k$ and their partial derivatives.
We consider two cases separately.

{\bf(1)} When $\operatorname{div}^0_g(\xi_0)=1$, then we know from the first equation in \eqref{Eq2.2} that $V_{z_j}(\xi_0)=0$ with multiplicity at least $1$, which then yields $U_{z_j}(\xi_0)=0$ with multiplicity at least $1$ by \eqref{Eq2.1}.
Since $f(\xi_0)\neq0$ via \eqref{Eq1.3} and $\Z\big(U_{z_2}/f\big)=\Z\big(V_{z_1}/g\big)$ ignoring multiplicities via \eqref{Eq1.5}, $\xi_0$ is a zero of $G$ by definition.
In view of \eqref{Eq2.12}, we have
\begin{equation}
\begin{split}
\operatorname{div}^0_{P_1}(\xi_0)&=\operatorname{div}^0_{F^2}(\xi_0)+\operatorname{div}^0_{G^2}(\xi_0)-2\operatorname{div}^0_{g^2}(\xi_0)\\
&=2\operatorname{div}^0_{U_{z_2}}(\xi_0)+2\operatorname{div}^0_{\frac{V_{z_1}}{g}}(\xi_0)-4\operatorname{div}^0_g(\xi_0)\geq0.\nonumber
\end{split}
\end{equation}

{\bf (2)} When $\operatorname{div}^0_g(\xi_0)=l\geq2$, then we derive from the first equation in \eqref{Eq2.2} that $V_{z_j}(\xi_0)=0$ with multiplicity at least $2l-1$, which immediately yields $U_{z_j}(\xi_0)=0$ with multiplicity at least $2l-1$ by \eqref{Eq2.1}.
As $f(\xi_0)\neq0$ by \eqref{Eq1.3}, we see from \eqref{Eq2.12} that
\begin{equation}
\begin{split}
\operatorname{div}^0_{P_1}(\xi_0)&=\operatorname{div}^0_{F^2}(\xi_0)+\operatorname{div}^0_{G^2}(\xi_0)-2\operatorname{div}^0_{g^2}(\xi_0)\\
&=2\operatorname{div}^0_{U_{z_2}}(\xi_0)+2\operatorname{div}^0_{\frac{V_{z_1}}{g}}(\xi_0)-4\operatorname{div}^0_g(\xi_0)\\
&\geq2(2l-1)+2(l-1)-4l=2(l-2)\geq0.\nonumber
\end{split}
\end{equation}

By virtue of \eqref{Eq2.12} and Claims 1-3, it is easily deduced that
\begin{equation}\label{Eq2.16}
N(r,P_1)\leq2\overline{N}(r,f)+O(T(r,h)+T(r,k)).
\end{equation}

When $P_1\equiv0$, then we observe from \eqref{Eq2.12} that $FG\equiv0$ again.
So, $U=hf^2$ and $V=kg^2$ are both constant as before, which yields the conclusion of Theorem \ref{Thm1.2}.

When $P_1\not\equiv0$, then we have from the analysis in \eqref{Eq2.5} and \eqref{Eq2.13}-\eqref{Eq2.16} that
\begin{equation}
\begin{split}
&\hspace{4.6mm}2T(r,f)=T\Big(r,\frac{1}{f^2}\Big)+O(1)=m\Big(r,\frac{1}{f^2}\Big)+2N\Big(r,\frac{1}{f}\Big)+O(1)\\
&\leq m\Big(r,\frac{P_1}{f^2}\Big)+m\Big(r,\frac{1}{P_1}\Big)+N\Big(r,\frac{1}{P_1}\Big)+O(T(r,h)+T(r,k))\\
&\leq T\Big(r,\frac{1}{P_1}\Big)+S(r,f)+S(r,g)+O(T(r,h)+T(r,k))\\
&=T(r,P_1)+S(r,f)+S(r,g)+O(T(r,h)+T(r,k))\\
&=m(r,P_1)+N(r,P_1)+S(r,f)+S(r,g)+O(T(r,h)+T(r,k))\\
&\leq2\overline{N}(r,f)+S(r,f)+S(r,g)+O(T(r,h)+T(r,k)),\nonumber
\end{split}
\end{equation}
since $P=\frac{P_1}{f^2}$ in view of \eqref{Eq2.3} with $m=n=2$ and \eqref{Eq2.12}.
That is,
\begin{equation}\label{Eq2.17}
\begin{split}
&\hspace{5mm}T(r,f)\leq\overline{N}(r,f)+S(r,f)+S(r,g)+O(T(r,h)+T(r,k))\\
&\leq(1-\Theta(\infty,f))T(r,f)+S(r,f)+S(r,g)+O(T(r,h)+T(r,k));
\end{split}
\end{equation}
in other words, recalling $\Theta(\infty,f)\leq1$, we have
\begin{equation}\label{Eq2.18}
\Theta(\infty,f)T(r,f)=O(T(r,h)+T(r,k)),
\end{equation}
which along with the condition $\Theta(\infty,f)>0$ yields the conclusion of Theorem \ref{Thm1.2}.
\end{proof}

\begin{remark}
A careful check of the analysis in Claim 1 of the proof of Theorem \ref{Thm1.2} {\bf(B)} indicates that one may only consider the simple poles of $f$, and thus weaken the condition $\Theta(\infty,f)>0$ and enlarge the family $\F$.
On the other hand, as equation \eqref{Eq1.3} with $m=n$ yields the estimate \eqref{Eq2.13} as well as
\begin{equation}
\overline{N}(r,f)=\overline{N}(r,g)+O(T(r,h)+T(r,k)),\nonumber
\end{equation}
one can alternatively choose $g\in\F$ in place of $f\in\F$ in Theorem \ref{Thm1.2} {\bf(B)}.
\end{remark}

\section{Some extensions}\label{Ext}
First, recall a constraint on the poles of $f$ (and automatically on those of $g$) was utilized to have the estimate \eqref{Eq1.2} in Theorem \ref{Thm1.2} {\bf(B)}.
Now, we would like to replace this constraint by a similar one.
Denote by $\Fa$ the set of meromorphic functions $f$ in $\C^2$ satisfying $\Theta(a,f)>0$ for $\Theta(a,f):=1-\limsup\limits_{r\to\infty}\frac{\overline{N}\big(r,\frac{1}{f-a}\big)}{T(r,f^2)}$.
All entire functions are members of $\F_\infty=\F$.

Our first extension result is as follows regarding meromorphic solutions to
\begin{equation}\label{Eq3.1}
h(z_1)f^2+k(z_1,z_2)g^2=1.
\end{equation}

\begin{theorem}\label{Thm3.1}
Assume that $h(z_1)\not\equiv0$ is a meromorphic function in $\C$, and $k(z_1,z_2)\not\equiv0$ is a meromorphic function in $\C^2$.
Then, all meromorphic solutions $f(z_1,z_2),g(z_1,z_2)$ to equation \eqref{Eq3.1} in $\C^2$ satisfy \eqref{Eq1.2} outside a set of $r$ of finite Lebesgue measure, provided
\begin{equation}
\Z(f_{z_2})=\Z\big((kg^2)_{z_1}/(2kg)\big)\nonumber
\end{equation}
ignoring multiplicities and $f\in\Fa$ for a finite number $a\neq0$ with $a^2h\not\equiv1$.
\end{theorem}

We  only outline the key steps of its proof, which resembles that of Theorem \ref{Thm1.2}.

\begin{proof}
Set $U=hf^2$ and $V=kg^2$ to have \eqref{Eq2.1}-\eqref{Eq2.2} as before.
Write
\begin{equation}\label{Eq3.2}
\begin{split}
&\hspace{4.68mm}H_a:=\frac{F^2}{f^2(f^2-a^2)}\frac{G^2}{g^4}=\frac{kF^2}{f^2(f^2-a^2)(1-hf^2)}\frac{G^2}{g^2}\\
&=k\bigg(\frac{1}{a^2(1-a^2h)}\frac{F^2}{f^2-a^2}-\frac{1}{a^2}\frac{F^2}{f^2}+\frac{h^2}{1-a^2h}\frac{F^2}{1-hf^2}\bigg)\frac{G^2}{g^2}
\end{split}
\end{equation}
for $F:=2hf_{z_2}=U_{z_2}/f$ and $G=k_{z_1}g+2kg_{z_1}=V_{z_1}/g$ like before.

We still have \eqref{Eq2.13}, and we apply the lemma on logarithmic derivative to get
\begin{equation}\label{Eq3.3}
m(r,H_a)=S(r,f)+S(r,g)+O(T(r,h)+T(r,k)).
\end{equation}
In view of $H_a=\frac{P_1}{f^2-a^2}$ by \eqref{Eq2.12} and \eqref{Eq3.2}, it follows that
\vskip 2pt
{\bf Claim 1.} $\operatorname{div}^0_{H_a}(\xi_0)\geq2l$ at each zero of $f$ with multiplicity $l$, which is neither a zero nor a pole of $h,k$ and their partial derivatives;
\vskip 2pt
{\bf Claim 2.} $\operatorname{div}^0_{H_a}(\xi_0)\geq0$ at each zero of $g$, which is not a zero of $f^2-a^2$, or equivalently not a zero of $h-a^{-2}$, and which is neither a zero nor a pole of $h,k$ and their partial derivatives.

In the sequel, we will consider the poles of $H_a$.

{\bf Claim 3.} $H_a$ is holomorphic at each pole of $f$ or $g$ that is neither a zero nor a pole of $h,k$ and their partial derivatives.
To prove this claim, rewrite $H_a$ as
\begin{equation}
H_a=\frac{F^2}{f^2(f^2-a^2)}\frac{G^2}{g^4}=\frac{4h^2\tilde{f}^2_{z_2}}{1-a^2\tilde{f}^2}(k_{z_1}\tilde{g}-2k\tilde{g}_{z_1})^2\nonumber
\end{equation}
for $\tilde{f}=\frac{1}{f}$ and $\tilde{g}=\frac{1}{g}$.
We observe from \eqref{Eq3.1} that a pole $\xi_\infty$ of $f$ with multiplicity $l$, which is neither a zero nor a pole of $h,k$ and their partial derivatives, is also a pole of $g$ with the same multiplicity $l$.
Therefore, we have $\operatorname{div}^0_{H_a}(\xi_\infty)\geq4(l-1)\geq0$.

Because poles of $f$ and $g$ appear simultaneously, Claim 3 is thus proved.

{\bf Claim 4.} $H_a$ has a pole at each zero of $f-a$, which is not a zero of $g$, and which is neither a zero nor a pole of $h,k$ and their partial derivatives, with multiplicity at most $1$.
To have this proved, we first note such a common zero of $f-a$ and $g$ is automatically a zero of $h-a^{-2}$ by \eqref{Eq3.1}.
Now, let $\xi_a$ be a zero of $f-a$ with multiplicity $l$ satisfying the given constraints.
Then, $\operatorname{div}^0_{H_a}(\xi_a)\geq2(l-1)-l=l-2\geq0$ if $l\geq2$; if $l=1$, then $\operatorname{div}^0_{\frac{1}{H_a}}(\xi_a)\leq1$.
Analogous analysis leads to that $H_a$ has a pole at each zero of $f+a$, which is not a zero of $g$, and which is neither a zero nor a pole of $h,k$ and their partial derivatives, with multiplicity at most $1$.

By virtue of \eqref{Eq3.1} and Claims 1-4, it is easy to have
\begin{equation}\label{Eq3.4}
N(r,H_a)\leq\overline{N}\Big(r,\frac{1}{f-a}\Big)+\overline{N}\Big(r,\frac{1}{f+a}\Big)+O(T(r,h)+T(r,k)).
\end{equation}

When $H_a\equiv0$, then $U=hf^2$ and $V=kg^2$ are both constant, which leads to the conclusion of Theorem \ref{Thm1.2}.
When $H_a\not\equiv0$, then like before one derives by \eqref{Eq3.3}-\eqref{Eq3.4} that
\begin{equation}
\begin{split}
&\hspace{4.6mm}2T(r,f)=T\Big(r,\frac{1}{f^2}\Big)+O(1)=m\Big(r,\frac{1}{f^2}\Big)+2N\Big(r,\frac{1}{f}\Big)+O(1)\\
&\leq m\Big(r,\frac{H_a}{f^2}\Big)+m\Big(r,\frac{1}{H_a}\Big)+N\Big(r,\frac{1}{H_a}\Big)+O(T(r,h)+T(r,k))\\
&\leq T(r,H_a)+S(r,f)+S(r,g)+O(T(r,h)+T(r,k))\\
&=m(r,H_a)+N(r,H_a)+S(r,f)+S(r,g)+O(T(r,h)+T(r,k))\\
&\leq\overline{N}\Big(r,\frac{1}{f-a}\Big)+\overline{N}\Big(r,\frac{1}{f+a}\Big)+S(r,f)+S(r,g)+O(T(r,h)+T(r,k)).\nonumber
\end{split}
\end{equation}
Noticing $\overline{N}\big(r,\frac{1}{f+a}\big)\leq T(r,f)+O(1)$, it follows that
\begin{equation}
\begin{split}
&\hspace{5mm}T(r,f)\leq\overline{N}\Big(r,\frac{1}{f-a}\Big)+S(r,f)+S(r,g)+O(T(r,h)+T(r,k))\\
&\leq(1-\Theta(a,f))T(r,f)+S(r,f)+S(r,g)+O(T(r,h)+T(r,k));\nonumber
\end{split}
\end{equation}
in other words, recalling $\Theta(a,f)\leq1$, we have
\begin{equation}\label{Eq3.5}
\Theta(a,f)T(r,f)=O(T(r,h)+T(r,k)),
\end{equation}
which along with the condition $\Theta(a,f)>0$ yields the conclusion of Theorem \ref{Thm3.1}.
\end{proof}

It is straightforward to have an analogous result on meromorphic solutions to
\begin{equation}\label{Eq3.6}
h(z_1,z_2)f^2+k(z_2)g^2=1.
\end{equation}

\begin{corollary}\label{Cor3.2}
Assume that $k(z_2)\not\equiv0$ is a meromorphic function in $\C$, and $h(z_1,z_2)\not\equiv0$ is a meromorphic function in $\C^2$.
Then, all meromorphic solutions $f(z_1,z_2),g(z_1,z_2)$ to equation \eqref{Eq3.6} in $\C^2$ satisfy \eqref{Eq1.2} outside a set of $r$ of finite Lebesgue measure, provided
\begin{equation}
\Z\big((hf^2)_{z_2}/(2hf)\big)=\Z(g_{z_1})\nonumber
\end{equation}
ignoring multiplicities and $g\in\Fa$ for a finite number $a\neq0$ with $a^2k\not\equiv1$.
\end{corollary}

In the sequel, we consider an extension of Corollary \ref{Cor1.3} as follows.

\begin{theorem}\label{Thm3.3}
All meromorphic solutions $u(z_1,z_2)$ to the general 2d eikonal equations
\begin{equation}\label{Eq3.7}
h(z_1)u^m_{z_1}+k(z_2)u^n_{z_2}=1,
\end{equation}
with $m,n\geq2$, satisfy the estimate \eqref{Eq1.2} outside a set of $r$ of finite Lebesgue measure.
Thus, if $h(z_1),k(z_2)$ are constant, then so are $u_{z_1},u_{z_2}$; that is, $u$ must be linear.
\end{theorem}

\begin{proof}
First, under the current assumption, one has
\begin{equation}
(hu^m_{z_1})_{z_2}/(mhu^{m-1}_{z_1})=u_{z_1z_2}=u_{z_2z_1}=(ku^n_{z_2})_{z_1}/(nku^{n-1}_{z_2});\nonumber
\end{equation}
as such, the conditions \eqref{Eq1.4} and \eqref{Eq1.5} are automatically satisfied.

When $m,n\geq2$ but $(m,n)\neq(2,2)$, then nothing is needed for further analysis.
In the case where $m=n=2$, let $\xi_\infty$ be a pole of $u_{z_1}$ or $u_{z_2}$, which is neither a zero nor a pole of $h,k$ and their derivatives.
Then, using \eqref{Eq3.7}, one sees that $\xi_\infty$ is a pole of both $u_{z_1}$ and $u_{z_2}$ of the same multiplicity, and thus a pole of $u$.
So, $\xi_\infty$ must be a multiple pole of both $u_{z_1}$ and $u_{z_2}$.
Noting the remark at the end of Section \ref{PMT}, the function $P_1$ by \eqref{Eq2.12} (after replacing $f,g$ by $u_{z_1},u_{z_2}$, respectively) only has possible poles from those of $h,k$, and therefore in this case, the deficient value condition is automatically satisfied, which finishes the proof.
\end{proof}

As we said earlier, when $h,k$ are constant, Khavinson \cite{Kh1} in 1995 proved the first result on entire solutions to the 2d eikonal equation for $m=n=2$, while Li \cite{Li05FM,Li08,Li12,Li14,Li18,LL,LY}, with Ye and L\"{u}, systemically studied its meromorphic solutions.
In this aspect, our Theorem \ref{Thm3.3} supplements those known results and provides the first one on meromorphic solutions with function coefficients $h,k$, especially in the situation where $m=n=2$.

Finally, we consider the relatively easier case of entire solutions to \eqref{Eq3.7} when either $m=1$ or $n=1$.
To proceed, we start with an example as follows.

\begin{example}
The entire function $u(z_1,z_2)=e^{z^2_1-z^3_2}+\frac{z^2_1}{2}+\frac{z^3_2}{2}$ in $\C^2$ is a solution to
\begin{equation}
h(z_1)u_{z_1}+k(z_2)u_{z_2}=1\nonumber
\end{equation}
for meromorphic functions $h(z_1)=\frac{1}{2z_1}$ and $k(z_2)=\frac{1}{3z^2_2}$ in $\C$.
Clearly, \eqref{Eq1.2} fails.
\end{example}

The two results in the sequel partially extends the studies of Li \cite[Corollary 2.7]{Li05TAMS} (see also Khavinson \cite{Kh1} and Hemmati \cite{He}) with function coefficients $h(z_1),k(z_2)$.

\begin{proposition}\label{Prop3.4}
All entire solutions $u(z_1,z_2)$ to the general 2d eikonal equations
\begin{equation}\label{Eq3.8}
h(z_1)u_{z_1}+k(z_2)u^n_{z_2}=1,
\end{equation}
satisfy the estimate \eqref{Eq1.2} outside a set of $r$ of finite Lebesgue measure for $n>2$.
On the other hand, when $n=2$ and if \eqref{Eq1.2} fails, then we have the following conditions:
\begin{equation}\label{Eq3.9}
\begin{split}
&\hspace{4.68mm}T(r,u_{z_1})=T(r,u_{z_1z_1})+S(r,u)\\
&=2T(r,u_{z_2})+S(r,u)=2T(r,u_{z_1z_2})+S(r,u)=2T(r,u_{z_2z_2})+S(r,u)
\end{split}
\end{equation}
and
\begin{equation}\label{Eq3.10}
T(r,u_{z_2})=N\Big(r,\frac{1}{u_{z_2}}\Big)+S(r,u)=\overline{N}\Big(r,\frac{1}{u_{z_2}}\Big)+S(r,u).
\end{equation}
\end{proposition}

\begin{proof}
First, for every entire solution $u$ to \eqref{Eq3.8} in $\C^2$, recall a well-known fact (via the lemma on logarithmic derivative) that says
\begin{equation}
\max\{T(r,u_{z_1}),T(r,u_{z_2})\}\leq T(r,u)+S(r,u).\nonumber
\end{equation}
On the other hand, via Li \cite[Theorem 3]{Li12} based on Chuang \cite{Ch} (for a detailed English presentation, see, for example, Theorem 4.1 of Yang \cite{Ya}; interested reader may also refer to Hayman \cite[Theorem 2]{Hay2}), one has, for some absolute constants $\tau(>1),c_\tau>0$,
\begin{equation}
T(r,u)\leq c_\tau\max\{T(\tau r,u_{z_1}),T(\tau r,u_{z_2})\}+O(\log r).\nonumber
\end{equation}
Note also, using \eqref{Eq3.8}, that
\begin{equation}\label{Eq3.11}
T(r,u_{z_1})=nT(r,u_{z_2})+O(T(r,h)+T(r,k)).
\end{equation}

From now on, suppose the estimate \eqref{Eq1.2} does not hold; in this case, in view of the preceding estimates, it follows that
\begin{equation}\label{Eq3.12}
T(r,h),T(r,k)=S(r,u).
\end{equation}

Next, take the partial derivative of $z_2$ on both sides of \eqref{Eq3.8} to get
\begin{equation}\label{Eq3.13}
h(z_1)u_{z_1z_2}=\{-k'(z_2)u_{z_2}-nk(z_2)u_{z_2z_2}\}u^{n-1}_{z_2};
\end{equation}
so, via Li \cite[Section 4]{Li96} (see also Hayman \cite[Lemma 3.3]{Hay1}), we have
\begin{equation}\label{Eq3.14}
T(r,k'u_{z_2}+nku_{z_2z_2})=S(r,u),
\end{equation}
and when $n>2$, we additionally have
\begin{equation}
T(r,\{k'u_{z_2}+nku_{z_2z_2}\}u_{z_2})=S(r,u),\nonumber
\end{equation}
both combined implying $T(r,u_{z_2})=S(r,u)$, and thus $T(r,u)=S(r,u)$ in view of \eqref{Eq3.11}-\eqref{Eq3.12} as well as the estimates in front of them, a contradiction.
Therefore, when $n>2$, the estimate \eqref{Eq1.2} must hold outside a set of $r$ of finite Lebesgue measure.

On the other hand, for $n=2$, one only has \eqref{Eq3.14}, which along with \eqref{Eq3.13} (that now writes $h(z_1)u_{z_1z_2}=\{-k'(z_2)u_{z_2}-2k(z_2)u_{z_2z_2}\}u_{z_2}$) and \eqref{Eq3.12} further yields
\begin{equation}
T(r,u_{z_2})=T(r,u_{z_1z_2})+S(r,u)=T(r,u_{z_2z_2})+S(r,u).\nonumber
\end{equation}
Now, take the partial derivative of $z_1$ on both sides of \eqref{Eq3.8} to get
\begin{equation}\label{Eq3.15}
h'(z_1)u_{z_1}+h(z_1)u_{z_1z_1}=-2k(z_2)u_{z_2}u_{z_2z_1}.
\end{equation}
Let $\xi_0=(\chi_1,\chi_2)$ be a zero of $u_{z_2}$ of multiplicity $l$, with $\chi_1,\chi_2$ the associated coordinates in $\C$, that is neither a zero nor a pole of $h,k$ and their derivatives.
So, one has $u_{z_1}(\xi_0)=1/h(\chi_1)\neq0$ by \eqref{Eq3.8} and $\operatorname{div}^0_{u_{z_1z_2}}(\xi_0)\geq2l-1$ by \eqref{Eq3.13}.
Substituting both into \eqref{Eq3.15} yields
\begin{equation}
\operatorname{div}^0_{u_{z_1z_1}+\frac{h'}{h^2}}(\xi_0)\geq3l-1\geq2l\nonumber
\end{equation}
by a slight abuse of notations.
As a result, using \eqref{Eq3.12} and \eqref{Eq3.14}, one observes
\begin{equation}
\begin{split}
&\hspace{4.68mm}T(r,u_{z_2})=T\bigg(r,\frac{k'u_{z_2}+2ku_{z_2z_2}}{k'+2k\frac{u_{z_2z_2}}{u_{z_2}}}\bigg)\\
&\leq T(r,k'u_{z_2}+2ku_{z_2z_2})+T\Big(r,k'+2k\frac{u_{z_2z_2}}{u_{z_2}}\Big)+O(1)\\
&\leq m\Big(r,\frac{u_{z_2z_2}}{u_{z_2}}\Big)+N\Big(r,\frac{u_{z_2z_2}}{u_{z_2}}\Big)+O(T(r,k))+S(r,u)\\
&\leq\overline{N}\Big(r,\frac{1}{u_{z_2}}\Big)+S(r,u)\leq N\Big(r,\frac{1}{u_{z_2}}\Big)+S(r,u)\\
&\leq\frac{1}{2}N\bigg(r,\frac{1}{u_{z_1z_1}+\frac{h'}{h^2}}\bigg)+S(r,u)\leq\frac{1}{2}T\Big(r,u_{z_1z_1}+\frac{h'}{h^2}\Big)+S(r,u)\\
&\leq\frac{1}{2}T(r,u_{z_1z_1})+\frac{1}{2}T\Big(r,\frac{h'}{h^2}\Big)+S(r,u)\\
&\leq\frac{1}{2}m(r,u_{z_1z_1})+O(T(r,h))+S(r,u)\\
&\leq\frac{1}{2}m(r,u_{z_1})+\frac{1}{2}m\Big(r,\frac{u_{z_1z_1}}{u_{z_1}}\Big)+S(r,u)=\frac{1}{2}T(r,u_{z_1})+S(r,u),\nonumber
\end{split}
\end{equation}
which together with \eqref{Eq3.11} forces all inequalities to equalities; so, $l=1$.
Therefore, when $n=2$ and if \eqref{Eq1.2} fails, then we arrive at the conditions \eqref{Eq3.9} and \eqref{Eq3.10}.
\end{proof}

\begin{corollary}\label{Cor3.5}
All entire solutions $u(z_1,z_2)$ to the general 2d eikonal equations
\begin{equation}\label{Eq3.16}
h(z_1)u^m_{z_1}+k(z_2)u_{z_2}=1,
\end{equation}
satisfy the estimate \eqref{Eq1.2} outside a set of $r$ of finite Lebesgue measure for $m>2$.
On the other hand, when $m=2$ and if \eqref{Eq1.2} fails, then we have the following conditions:
\begin{equation}\label{Eq3.17}
\begin{split}
&\hspace{4.68mm}T(r,u_{z_2})=T(r,u_{z_2z_2})+S(r,u)\\
&=2T(r,u_{z_1})+S(r,u)=2T(r,u_{z_1z_1})+S(r,u)=2T(r,u_{z_1z_2})+S(r,u)
\end{split}
\end{equation}
and
\begin{equation}\label{Eq3.18}
T(r,u_{z_1})=N\Big(r,\frac{1}{u_{z_1}}\Big)+S(r,u)=\overline{N}\Big(r,\frac{1}{u_{z_1}}\Big)+S(r,u).
\end{equation}
\end{corollary}

{\small{\bf Acknowledgments.}
Qi Han is supported by Texas A\&M University-San Antonio Research Council Grant \#21870120011 and the 2021 Faculty Research Fellowship from College of Arts and Sciences at Texas A\&M University-San Antonio.
W. Chen and Q. Wang are supported by the Science and Technology Research Program of Chongqing Municipal Education Commission \#KJQN202000621, the Fundamental Research Funds of Chongqing University of Posts \& Telecommunications \#CQUPT:A2018-125, and the Basic and Advanced Research Project of CQ CSTC \#cstc2019jcyj-msxmX0107.}

\end{document}